\documentclass{amsart}
\usepackage{mathrsfs}
\usepackage{amsmath,amssymb}

\usepackage{amsfonts}
\usepackage{amssymb,latexsym,amsthm}
\usepackage{hyperref}
\usepackage{amscd}
\usepackage{graphicx} 
\usepackage{graphicx,epstopdf,subfigure}
\usepackage{epsfig}
\usepackage{subfigure}
\usepackage[all]{xy}
\usepackage{xcolor}
\usepackage{framed}
\hypersetup{linkcolor=blue, pdfstartview=FitH} \hypersetup{colorlinks, linkcolor=blue,
pdfstartview=FitH}
\newtheorem{corollary}{Corollary}
\newtheorem{lemma}[corollary]{Lemma}
\newtheorem{definition}[corollary]{Definition}
\newtheorem{proposition}[corollary]{Proposition}
\newtheorem{theorem}[corollary]{Theorem}

\date{}

\begin{document}

\title[A Gaussian version of Littlewood's theorem]{{\small \textbf{A Gaussian version of Littlewood's theorem on random power series }}}
\author{{\small
{\large{G}}UOZHENG {\large{C}}HENG,
{\large{X}}IANG {\large{F}}ANG,
{\large{K}}UNYU {\large{G}}UO
AND {\large{C}}HAO {\large{L}}IU}}

\maketitle
\def\cc{\mathbb{C}}
\def\zz{\mathbb{Z}}
\def\nn{\mathbb{N}}
\def\rr{\mathbb{R}}
\def\qq{\mathbb{Q}}
\def\dd{\mathbb{D}}
\def\tt{\mathbb{T}}
\def\bb{\mathbb{B}}
\def\ff{\mathbb{F}}

\def\A{\mathcal{A}}
\def\B{\mathcal{B}}
\def\D{\mathcal{D}}
\def\L{\mathcal{L}}
\def\M{\mathcal{M}}
\def\Mperp{\mathcal{M}^\perp}
\def\N{\mathcal{N}}
\def\K{\mathcal{K}}
\def\F{\mathcal{F}}
\def\R{\mathcal{R}}
\def\s{\mathcal{S}}
\def\p{\mathcal{P}}
\def\P{\mathcal{P}}
\def\T{\mathcal{T}}
\def\O{\mathcal{O}}
\def\Z{\mathcal{Z}}
\def\om{\omega}
\def\Om{\Omega}
\def\cod{\text{cod}}
\def\Cov{\text{Cov}}
\def\ker{\text{ker}}
\def\Ker{\text{Ker}}
\def\Var{\text{Var}}
\def\mp{\text{mp}}

\def\al{\alpha}
\def\la{\lambda}
\def\ep{\epsilon}
\def\sig{\sigma}
\def\Sig{\Sigma}
\def\bx{\mathcal{X}}
\def\by{\mathcal{Y}}
\def\bm{\mathcal{M}}
\def\bn{\mathcal{N}}
\def\ba{\mathcal{A}}
\def\be{\mathcal{E}}
\def\dd{\mathbb{D}}
\def\hr{H^2({\mathbb D}^2)}
\def\bigpa#1{\biggl( #1 \biggr)}
\def\bigbracket#1{\biggl[ #1 \biggr]}
\def\bigbrace#1{\biggl\lbrace #1 \biggr\rbrace}

\newenvironment{note}{\begin{tcolorbox}
\noindent

}{\end{tcolorbox}}

\def\papa#1#2{\frac{\partial #1}{\partial #2}}
\def\dbar{\bar{\partial}}

\def\oneover#1{\frac{1}{#1}}

\def\meihua{\bigskip \noindent $\clubsuit \ $}
\def\lingxing{\bigskip \noindent $\blacklozenge\ $}
\def\blue#1{\textcolor[rgb]{0.00,0.00,1.00}{#1}}
\def\red#1{\textcolor[rgb]{1.00,0.00,0.00}{#1}}

\def\norm#1{||#1||}
\def\inner#1#2{\langle #1, \ #2 \rangle}

\def\bigno{\bigskip \noindent}
\def\medno{\medskip \noindent}
\def\smallno{\smallskip \noindent}
\def\bignobf#1{\bigskip \noindent \textbf{#1}}
\def\mednobf#1{\medskip \noindent \textbf{#1}}
\def\smallnobf#1{\smallskip \noindent \textbf{#1}}
\def\nobf#1{\noindent \textbf{#1}}
\def\nobfblue#1{\noindent \textbf{\textcolor[rgb]{0.00,0.00,1.00}{#1}}}

\def\vector#1#2{\begin{pmatrix}  #1  \\  #2 \end{pmatrix}}

\def\devide{\bigskip \hrule
\bigno}

\def\divides{\bigskip \hrule

\bigno}
%
%
%

\begin{abstract}
We prove a Littlewood-type theorem on random analytic functions for not necessarily independent Gaussian processes.
We show that if we randomize a function in the Hardy space $H^2(\dd)$ by a Gaussian process whose  covariance matrix $K$ induces a bounded operator on $l^2$, then the resulting random function is almost surely in $H^p(\dd)$ for any $p>0$.
The case  $K=\text{Id}$, the identity operator, recovers Littlewood's theorem.
A new ingredient in our proof is to recast the membership problem as the boundedness of an operator. This reformulation enables us to use tools in functional analysis and is applicable to other situations. The sharpness of the new condition and several ramifications are discussed.
\end{abstract}
\footnote{


\noindent 
\textbf{Mathematics Subject Classification [MSC]:} 30B20; 47B38.}



\section{Introduction and main results}
\noindent
A folklore about the summation of a series is that a randomized series often enjoys improved regularity. A well known example is that $\sum_{n=1}^\infty \pm \frac{1}{n^p}$ is almost surely convergent if and only if $p > \frac{1}{2}.$
In the setting of random analytic functions, one of the best known results is  Littlewood's theorem. Let $f(z)=a_0+a_1z+a_2z^2+\cdots \in H^2(\dd)$ be an element of the Hardy space over the unit disk. Let $\{\epsilon_n\}_{n=0}^\infty$ be a sequence of independent, identically distributed Rademacher random variables, that is, $P(\epsilon_n=1)=P(\epsilon_n=-1)=\frac{1}{2}$ for all $n\geq0.$  Littlewood's theorem, proved in 1930 \cite{Littlewood},  states that
$$(\mathcal{R}f)(z)\doteq\sum_{n=0}^\infty a_n \epsilon_n z^n \in H^p(\dd)$$
almost surely for all $p \ge 2.$
The same is true for a standard Steinhaus sequence  \cite{Littlewood1,Paley1} and for a standard Gaussian sequence (\cite{RandomSeries}, p. 54; \cite{PWZ1933}).  Littlewood's theorem can be restated as that $\mathcal{R}f$ represents  $H^p(\dd)$ functions almost surely if and only if $f \in H^2(\dd)$.

\bigno When the random  series
$
\sum_{n=0}^{\infty}a_n\epsilon_nz^n
$
represents a function in $H^\infty(\dd)$ almost surely is much harder. Necessary conditions and sufficient conditions were given by
Paley and Zygmund \cite{Paley1} and later by Salem and Zygmund \cite{Salem}. In 1963,   Billard \cite{Billard63} proved that the problem is equivalent to the case of  Steinhaus sequences. A remarkable characterization was finally obtained by  Marcus and Pisier in 1978 \cite{Pisier78} (see also \cite{RandomSeries,Pisier1}). The proof relies on the celebrated Dudley-Fernique theorem.
%

 \bigno In 1974
 Anderson, Clunie and Pommerenke \cite{JJC} studied the case of the Bloch space $\mathcal{B}$.
A necessary and sufficient condition for a random Taylor series to represent a Bloch function is given by Gao \cite{Gao} in 2000.
 In 1981, Sledd \cite{Sledd} showed that the condition $
\sum_{n=1}^{\infty}|a_n|^2\log n<\infty
$  implies that
$
\sum_{n=0}^{\infty}\epsilon_na_nz^n\in\text{BMOA a.s.}
$
Then  Stegenga \cite{Stegenga} showed that there is a sequence $\{a_n\}_{n=0}^{\infty}\in\ell^2$ but
$
\sum_{n=0}^{\infty}\epsilon_na_nz^n\notin\text{BMOA a.s.}
$
  In \cite{PDuren1}, Duren explored the difference between BMOA and $\bigcap_{0<p<\infty}H^p$ for the random Taylor series. For convenience, we  write $H^p$ for $H^p(\dd)$. More recent results on random BMOA functions are due to Nishry-Paquette \cite{Nishry2020}.
 Cochran, Shapiro and Ullrich proved in 1993 that a  Dirichlet function with random signs is a.s. a Dirichlet multiplier \cite{WJD}. Along a somehow different direction, random Dirichlet series was investigated, say, by Konyagin, Queff\'{e}lec, Saksman and Seip \cite{KQSS} for a recent reference.

\bigno
Our interest in this subject is largely due to \cite{CFZ} where a random version of the  Hardy space $H^2$ is introduced. Then \cite{ChengFangLiu} studies Littlewood-type phenomena for random Bergman functions, which, however, exhibit no improvement of regularity for any $p>0$. In \cite{ChengFangLiu}, it is shown that
if $p\ge 2$ and $f \in L^p_a(\dd)$, then $\R f \in L^q_a(\dd)$ almost sure if and only if $q\le p$.  Moreover,  if $0<p<2$, then $\R f \in L^q_a(\dd)$ almost surely for all $f
\in L^p_a(\dd)$ if and only if $\frac{1}{q} >\frac{2}{p}-\frac{1}{2}.$ This sharp contrast with the Hardy spaces prompts us to  take a deeper look at the original Littlewood theorem.

\bigno A close examination of known proofs of either Littlewood's  arguments \cite{Littlewood} or later improvement by Paley and Zygmund \cite{Paley1}, or the polished proofs in the monographs of Duren \cite{Duren} and Kahane \cite{RandomSeries}, reveals that the independence among the randomizing sequence $\{\epsilon_n\}_{n=0}^{\infty}$ renders one to sum up the contribution of each monomial separately. In other words, after independent randomization, the contribution of different $z^n$'s to the $H^p$-norm is simply added up via a triangle inequality. This is indeed a common pattern in the existing literature. It is perhaps unsatisfactory  and suggests room for improvement to non-independent randomization.

\bigno
In this note we try to relate the $H^p$-structure to  the correlation of a random process.
Although this   appears to be a new  effort from the viewpoint of random power series, it is natural from the viewpoint of Gaussian analytic functions (GAF),  a beautiful subject lying at the intersection of complex analysis and probability. GAF receives considerably attention in recent years but it progresses along different directions which we shall briefly review later; in particular, Littlewood-type phenomena remain largely overlooked.

\bigno  By a Gaussian analytic function (GAF) over the unit disk one   means a measurable map
$F: (\Omega, \mathcal{F}, \mathbb{P}) \to \text{Hol}(\dd)$ from a probability space to the space of
analytic functions on the unit disk such that $F(\cdot)(z)$ is a Gaussian variable for each $z \in \dd$. Here $\text{Hol}(\dd)$ is a Fr\'{e}chet space with the topology of uniform convergence
on compact subsets of $\dd$. The latter topology is   metrizable and the measurable structure on $\text{Hol}(\dd)$
is  naturally   the $\sigma$-field generated by the Borel topology. By a slight abuse of notations we write
$F_\omega(z)\doteq (F(\omega))(z)$ which is jointly measurable in $\omega$ and $z$. One may take the formal power series expansion
$F_\omega (z) =\sum_{n=0}^{\infty} X_n z^n$. Then it is easily shown that $X_n$ is measurable and $\{X_n\}_{n=0}^\infty$ form a Gaussian
process, clearly not necessarily of independent entries.

\begin{definition} A sequence of random variables $\{X_n\}_{n=0}^{\infty}$ has the \textbf{$L_p$-property} if for any
$
f(z)=\sum_{n=0}^{\infty}a_nz^n\in H^2,
$
one has
\begin{equation}
\big(\mathcal{R}f\big)(z)\doteq\sum_{n=0}^{\infty}a_nX_nz^n\in H^p \ \ \text{a.s.}
\end{equation}
 Moreover, $\{X_n\}_{n=0}^{\infty}$ has the \textbf{L-property} if  it has $L_p$ for all $p >0$.
\end{definition}

\noindent
In this note only real Gaussian processes are concerned.
Recall that the covariance matrix of a Gaussian process $\{X_n\}_{n=0}^{\infty}$ is
\begin{equation}
K=\Big(\mathbb{E}\big((X_m-\mathbb{E}X_m)(X_n-\mathbb{E}X_n)\big)\Big)_{n, m\geq0}.
\end{equation}
A Gaussian process $\{X_n\}_{n=0}^{\infty}$ is called centered if $\mathbb{E}X_n=0, \ n\geq0.$ It is a basic fact that the law of a centered Gaussian process is  determined by its covariance matrix \cite{Hida}. The main result of this note is

\begin{theorem}\label{T:main}
If the covariance matrix $K$ of  a centered Gaussian process $\{X_n\}_{n=0}^{\infty}$ is bounded on $\ell^2,$ then
$\{X_n\}_{n=0}^{\infty}$ has the $L$-property.
\end{theorem}

\noindent
For example, if $K=\big(\frac{1}{i+j+1}\big)_{i,j\geq0},$ the Hilbert matrix, then $\{X_n\}_{n=0}^{\infty}$ has the $L$-property. Several other examples are given in Section \ref{S:E}. By well known results (\cite{RandomSeries}, pp. 22,  54 and  179), the case $K=\text{Id}$, the identity operator, recovers the original Littlewood theorem.

\bigno
The sharpness of Theorem  \ref{T:main} is discussed in Subsection \ref{SS:sharp}. To complement Theorem \ref{T:main}, we offer the following three results before we end this introduction.
The first one is the reduction to the mean zero case.

\begin{lemma}\label{reductionzero}
Let  $\{X_n\}_{n=0}^{\infty}$ be a Gaussian process, not necessarily independent.
 Then $\{X_n\}_{n=0}^{\infty}$ has the $L_p$-property $(p>0)$ if and only if
 $\{X_n-\mathbb{E}X_n\}_{n=0}^{\infty}$ has the $L_p$-property and
$
\{\mathbb{E}X_n\}_{n=0}^{\infty} $  is a coefficient multiplier from $H^2$ to $H^p$.
\end{lemma}

 \noindent \emph{Remark.}
The coefficient multiplier space $(H^2, H^p)$ is known $(H^2, H^p)=\ell^\infty$ if $0<p\leq2$ and $(H^2, H^\infty)=H^2$ (\cite{2016Taylor}, pp. 215, 261). The characterization of $(H^2, H^p)$ for $2<p<\infty$ is still an open problem  (\cite{2016Taylor}, p. 276).

\bigno
The next result is analogous to  Paley and Zygmund's exponential improvement  \cite{Paley1} of Littlewood's theorem.
 \begin{proposition}\label{exponentintegral}
Assume that the covariance matrix $K$ of a centered Gaussian process $\{X_n\}_{n=0}^{\infty}$ is bounded on $\ell^2.$ Then for each $\{a_n\}_{n=0}^{\infty}\in\ell^2,$ $\sum_{n=0}^{\infty}a_nX_ne^{in\theta}$ converges a.s. almost everywhere, and for any $\la>0,$
$$\int_{0}^{2\pi}\exp\Big(\la\Big|\sum_{n=0}^{\infty}a_nX_ne^{in\theta}\Big|^2\Big)d\theta<\infty \ \ \text{a.s.}$$
\end{proposition}

\noindent Lastly, the independent case admits the following (slight) generalization.

\begin{proposition}\label{independentLp} Let  $\{X_n\}_{n=0}^{\infty}$ be a sequence of independent Gaussian variables with
$$
\mathbb{E}X_n=\mu_n, \ \ \text{Var}\big(X_n\big)=\sigma_n^2, \ \ n\geq0.
$$
 Then $\{X_n\}_{n=0}^{\infty}$ has the $L_p$-property for $p>0$  if and only if
\begin{itemize}
\item[(1)] $\{\sigma_n\}_{n=0}^{\infty} \in\ell^\infty$, and
\item[(2)] $\{\mu_n\}_{n=0}^{\infty} $ is a coefficient multiplier from $H^2$ to $H^p$.
\end{itemize}
\end{proposition}

\noindent Another generalization of the $L_p$-property (to compact groups) is given by Helgason \cite{Helgason57}, Fig\`{a}-Talamanca and Rider \cite{Alessandro66}. A related line of research, loosely under the umbrella of Gaussian analytic functions (GAF) is in
 Peres-Vir\'{a}g \cite{Peres05}, Sodin \cite{Sodin10,Sodin00,Sodin05}, or
 the monograph by Hough-Krishnapur-Peres-Vir\'{a}g \cite{Peres09}.
For general properties about the regularity of Gaussian processes, readers can refer to  Talagrand \cite{Talagrand14}
 or Ledoux-Talagrand \cite{91LedouxTalagrand}.
For more properties about $H^p,$ see \cite{Duren,Garnett,Hoffman62,Koosis}.
For more about random Taylor series, see \cite{Kahane97,Kahane997}.

\section{Preliminary issues: Convergence radius, the independent case and reduction to mean zero}

\noindent Let  $\{X_n\}_{n=0}^\infty$ be  a centered Gaussian process with $
  \text{Var}(X_n)=\sigma_n^2$ $(n \ge 0)$.  We first observe that for  $\{a_n\}_{n=0}^{\infty}\in\ell^2,$
  the convergence radius of
  $
  (\mathcal{R}f)(z)=\sum_{n=0}^{\infty}a_nX_nz^n
  $
   is at least one almost surely if and only if $ \limsup\limits_{n\to\infty}|\sigma_n|^{1/n}\leq1$.  This is a consequence of  the following fact whose  proof is a standard application of the Borel-Cantelli lemma, hence skipped. Note that $\{X_n\}$'s are not necessarily independent.

\begin{lemma}\label{radiusGaussian}
Let $\{X_n\}_{n=0}^{\infty}$ be a sequence of standard Gaussian variables, i.e.,  $X_n \sim N(0, 1)$ for any $n=0, 1, 2, \cdots$. Then
$
\lim_{n\to\infty}|X_n|^{1/n}=1 \ \text{a.s.}
$
\end{lemma}

\bigno
We first prove Proposition \ref{independentLp}. Recall that
\begin{equation}
  \big(H^2, E\big)=\Big\{\{\la_n\}_{n=0}^{\infty}: \ \sum_{n=0}^{\infty}a_n\la_nz^n\in E,\ \text{for every $f(z)=\sum_{n=0}^{\infty}a_nz^n\in H^2$} \Big\},
  \end{equation}
  where $E$ is an analytic function space.

\noindent
\begin{proof}[Proof of Proposition \ref{independentLp}:]
Let $
\{\sigma_n\}_{n=0}^{\infty}\in\ell^\infty \ \ \text{and} \ \
\{\mu_n\}_{n=0}^{\infty}\in\big(H^2, H^p\big),$ and let $
  X_n=\sigma_n\xi_n+\mu_n, \ n\geq0,$
  where $\xi_n$ are independent standard Gaussian variables.
For 
  $f(z)=\sum_{n=0}^{\infty}a_nz^n\in H^2,
$ we write
 $
  \big(\mathcal{R}f\big)(z) =
  \sum_{n=0}^{\infty}a_n\sigma_n\xi_nz^n+\sum_{n=0}^{\infty}a_n\mu_nz^n,
  $
  which converges pointwisely in the unit disk $\mathbb{D}$ by Lemma \ref{radiusGaussian}. An application of Littlewood's theorem and the fact $(H^2, H^2)=\ell^{\infty}$ yields that
$\sum_{n=0}^{\infty}a_n\sigma_n\xi_nz^n\in H^p$ a.s.
Since
 $
\{\mu_n\}_{n=0}^{\infty}\in\big(H^2, H^p\big),
$
  we have
$
\sum_{n=0}^{\infty}a_n\mu_nz^n\in H^p.$
%
On the other hand, if $\{X_n\}_{n=0}^{\infty}$ has the $L_p$-property, then by symmetry,
$$
\sum_{n=0}^{\infty}a_n\big(\sigma_n(-\xi_n)+\mu_n\big)z^n
\in H^p \ \text{a.s.}
$$
So
$
\sum_{n=0}^{\infty}a_n\mu_nz^n\in H^p.
$
 Now assume that
$
\{\sigma_n\}_{n=0}^{\infty}\notin\ell^\infty.
$
There exists a function
$
g(z)=\sum_{n=0}^{\infty}b_nz^n\in H^2
$
such that
$
\sum_{n=0}^{\infty}b_n\sigma_nz^n\notin H^2.
$
 It follows that
$
\sum_{n=0}^{\infty}|b_n\sigma_n|^2=\infty.
$
 By Littlewood's theorem, we have
$
\sum_{n=0}^{\infty}b_n\sigma_n\xi_nz^n\notin\bigcup_{0<p<\infty}H^p \ \ \text{a.s.}
$
On the other hand,
$$
\sum_{n=0}^{\infty}b_n\sigma_n\xi_nz^n=\sum_{n=0}^{\infty}b_nX_nz^n-
\sum_{n=0}^{\infty}b_n\mu_nz^n
\in H^p \ \ \text{a.s.}
$$
This leads to a contradiction.
\end{proof}


\noindent
\begin{proof}[Proof of Lemma \ref{reductionzero}:]
It is sufficient to show that if $\sum_{n=0}^{\infty}a_nX_nz^n \in H^p \ \text{a.s.}$,  then $
\{\mathbb{E}X_n\}_{n=0}^{\infty} \in\big(H^2, H^p\big).$
Since  $\{X_n-\mathbb{E}X_n\}_{n=0}^{\infty}$ is  a symmetric process,  meaning that  $\{X_n-\mathbb{E}X_n\}_{n=0}^{\infty}$ has the same law as
the process $\{-(X_n-\mathbb{E}X_n)\}_{n=0}^{\infty},$ it follows that
$
\sum_{n=0}^{\infty}a_n\big(-(X_n-\mathbb{E}X_n)+\mathbb{E}X_n\big)z^n\in H^p \ \text{a.s.}
$
So $\sum_{n=0}^{\infty}a_nX_nz^n \in H^p \ \text{a.s.}$ implies that
$
\sum_{n=0}^{\infty}a_n(\mathbb{E}X_n)z^n\in H^p.
$
\end{proof}

\section{A new formulation and  proof of the main result}

\noindent
In this section we  prove the following Theorem \ref{mainresult} which covers Theorem \ref{T:main} in the introduction. As mentioned before, a new viewpoint is  part (1)  which recasts the classical membership problem as the boundedness of an operator. This enables us to use tools in functional analysis, including the principle of uniform boundedness and Bessel sequences. This  sheds some new light on  this area, even in the case of independent random variables. We let $L^p(\Omega, H^q) \ (1\leq p, q<\infty)$ denote the collection of random vectors in $H^q$, i.e., $X: \Omega\to H^q$, such that $$||X||_{L^p(\Omega, H^q)}^p=\mathbb{\mathbb{E}}||X||_{H^q}^p<\infty.$$

\begin{theorem}\label{mainresult}
Let $\{X_n\}_{n=0}^{\infty}$ be a centered Gaussian process defined on a probability space $(\Omega, \mathcal{F}, P).$
\begin{itemize}
  \item[(1)] Let $p\in[1, \infty).$ Then for all $f(z)=\sum_{n=0}^{\infty}a_nz^n\in H^2,$ $\mathcal{R}f\in H^p$ a.s. if and only if $\mathcal{R}: H^2\to L^2(\Omega, H^p)$ is bounded.
  \item[(2)]If the covariance matrix $K$ of the centered Gaussian process $\{X_n\}_{n=0}^{\infty}$ is bounded on $\ell^2,$ then
$\mathcal{R}: H^2\to L^2(\Omega, H^p)$ is bounded for every $p\geq1.$
\end{itemize}
\end{theorem}

\noindent We first  illustrate the ideas in the independent case. Let  $f(z)=\sum_{k=0}^{n}a_kz^k.$
By the Kahane-Khintchine inequality (\cite{Tuomas17}, Theorem 6.2.4),
 for each $1\leq q<\infty,$ we have
\begin{equation}\label{KKinequality}
\bigg|\bigg|\sum_{k=0}^{n}a_k\epsilon_kz^k\bigg|\bigg|_{L^q(\Omega, H^p)}\leq\kappa_{p,q}\bigg|\bigg|\sum_{k=0}^{n}a_k\epsilon_kz^k\bigg|\bigg|_{L^p(\Omega, H^p)}.
\end{equation}
Then, using Fubini's theorem and  the Kahane-Khintchine inequality again, we have
$$
\bigg|\bigg|\sum_{k=0}^{n}a_k\epsilon_kz^k\bigg|\bigg|_{L^p(\Omega, H^p)}
=\bigg(\int_{0}^{2\pi}\mathbb{E}\bigg|\sum_{k=0}^{n}a_k\epsilon_ke^{ik\theta}\bigg|^p\frac{d\theta}{2\pi}\bigg)^{1/p} \leq C_p\bigg(\sum_{k=0}^{n}|a_k|^2\bigg)^{1/2}.
$$
It follows that
$$
\bigg|\bigg|\sum_{k=0}^{n}a_k\epsilon_kz^k\bigg|\bigg|_{L^q(\Omega, H^p)}\leq C_{p,q}\bigg(\sum_{k=0}^{n}|a_k|^2\bigg)^{1/2}
$$ for some constant $C_{p,q}$.
By a density argument, we  conclude that
$
\mathcal{R}: \ H^2\to L^q(\Omega, H^p)
$
is bounded. In other words, we can prove that  the following are equivalent:

\begin{itemize}
  \item[(i)]
  $
  \mathcal{R}: \ H^2\to L^q(\Omega, H^p)
$
is bounded for every $1\leq q<\infty.$
  \item[(ii)]
  $
  \mathcal{R}: \ H^2\to L^q(\Omega, H^p)
$
is bounded for some $q \ (1\leq q<\infty).$
  \item[(iii)]
 $
\mathcal{R}f\in H^p \ \ \text{a.s. for all $f\in H^2.$}
$
\end{itemize}

\noindent \emph{Remark.}
  If  $\mathcal{R}$ is given by a standard Gaussian sequence, then the norm of $
\mathcal{R}: \ H^2\to L^q(\Omega, H^p)
$ admits the following estimates:
\[
||\mathcal{R}||\leq\left\{
\begin{array}{ll}
2\Big(\frac{1}{\sqrt{\pi}}
\Gamma\big(\frac{q+1}{2}\big)\Big)^{1/q} \ & \ 1\leq p\leq q;\\
2\Big(\frac{1}{\sqrt{\pi}}
\Gamma\big(\frac{p+1}{2}\big)\Big)^{1/p} \ & \ 1\leq q\leq p.
\end{array}
\right.
\]
The   constants are derived from  Corollary 3 in \cite{99Latala}. For   Rademacher  or Steinhaus sequences, similar estimates are available.

\

\noindent
 Now for the convenience of the reader, we collect some facts before  the proof of Theorem \ref{mainresult}.  By Fernique's theorem (\cite{RandomSeries},  p. 176)
or by Corollary 3.2 in \cite{91LedouxTalagrand}, together with  Corollary 5.3 in \cite{02Latala},  we have

\begin{lemma}\label{equivalentmoment}
Let $\{X_n\}_{n=0}^\infty$  be a centered Gaussian process satisfying the {$L_p$-property} for some $1\leq p<\infty$. Then for $1\leq q<\infty$ and for each $\{a_n\}_{n=0}^{\infty}\in\ell^2,$
we have
\begin{itemize}
  \item[(i)] $1\leq q<2,$
  $$
  \bigg|\bigg|\sum_{n=0}^{\infty}a_nX_nz^n\bigg|\bigg|_{L^q(\Omega, H^p)}\leq
  \bigg|\bigg|\sum_{n=0}^{\infty}a_nX_nz^n\bigg|\bigg|_{L^2(\Omega, H^p)}\leq
  \frac{2}{c_q}\bigg|\bigg|\sum_{n=0}^{\infty}a_nX_nz^n\bigg|\bigg|_{L^q(\Omega, H^p)};
  $$
  \item[(ii)] $q\geq2,$
  $$
  \bigg|\bigg|\sum_{n=0}^{\infty}a_nX_nz^n\bigg|\bigg|_{L^2(\Omega, H^p)}\leq
  \bigg|\bigg|\sum_{n=0}^{\infty}a_nX_nz^n\bigg|\bigg|_{L^q(\Omega, H^p)}\leq
  2c_q\bigg|\bigg|\sum_{n=0}^{\infty}a_nX_nz^n\bigg|\bigg|_{L^2(\Omega, H^p)},
  $$
\end{itemize}
where
 $
c_q=\big(\mathbb{E}|\xi|^q\big)^{1/q}=\sqrt{2}\Big(\frac{1}{\sqrt{\pi}}
\Gamma\big(\frac{q+1}{2}\big)\Big)^{1/q}
$
and $\xi$ satisfies the standard Gaussian distribution.

\end{lemma}

\noindent
Recall that a sequence $\{X_k\}_{k=0}^{\infty}$
 in a separable Hilbert space $\mathcal{H}$ is called \textbf{a Bessel sequence} if there exists $C>0$ such that
$ \sum_{k=0}^{\infty}\big|\langle X, X_k\rangle\big|^2\leq C||X||^2 $ for all $ X\in\mathcal{H}.$
The constant $C=C_{\text{BES}}$   is called a Bessel bound for the Bessel sequence $\{X_k\}_{k=0}^{\infty}$.

\begin{proposition}[\cite{Christensen}]\label{Bessel}
Let $\{X_k\}_{k=0}^{\infty}$ be a sequence of real Gaussian variables in $L^2(\Omega)$ and $C>0.$ Then the following statements are equivalent.
\begin{itemize}
  \item[(i)] $\{X_k\}_{k=0}^{\infty}$ is a Bessel sequence in $L^2(\Omega)$ with a Bessel bound $C_{\text{BES}}.$
  \item[(ii)]
  $
  \sum_{k=0}^{\infty}\la_kX_k
  $
  converges in $L^2(\Omega)$ for  $\la=\{\la_k\}_{k=0}^{\infty}\in\ell^2$ and
  $\big\|\sum_{k=0}^{\infty}\la_kX_k\big\|\leq \sqrt{C_{\text{BES}}}||\la||_2.$
  \item[(iii)] The infinite Gram matrix $\big(\mathbb{E}(X_iX_j)\big)_{i,j\geq0}$ defines a bounded linear operator on $\ell^2$ with norm at most $C_{\text{BES}}.$
\end{itemize}
\end{proposition}

\noindent
\begin{proof}[Proof of (1) of Theorem \ref{mainresult}:]
It is sufficient to prove the necessity.   Consider $1\leq p\leq2$ first. For each $0<r<1,$ let
$
\mathcal{R}^{(r)}: H^2\to L^2(\Omega, H^p)
$ be
$
\big(\mathcal{R}^{(r)}f\big) (z)=\sum_{n=0}^{\infty}a_nr^nX_nz^n.
$
Since $\sum_{n=0}^{\infty}a_nX_nz^n\in H^p$ a.s., by Fernique's theorem,
\begin{equation}\label{Fernique111}
\exp\bigg(\la\bigg|\bigg|\sum_{n=0}^{\infty}a_nX_nz^n\bigg|\bigg|_{H^p}^2\bigg)\in L^1(\Omega)
\end{equation}
for some $\la>0.$
Then, by Minkowski's inequality and the Cauchy-Schwarz inequality,
\begin{eqnarray*}
||\mathcal{R}^{(r)}f||_{L^2(\Omega, H^p)}&\leq&\bigg(\int_{0}^{2\pi}\bigg(\mathbb{E}\bigg|\sum_{n=0}^{\infty}a_nr^nX_ne^{in\theta}\bigg|^2\bigg)^{p/2}\frac{d\theta}{2\pi}\bigg)^{1/p}\\
&\leq&\bigg(\sum_{n=0}^{\infty}r^{2n}\mathbb{E}X_n^2\bigg)^{1/2}\bigg(\sum_{n=0}^{\infty}|a_n|^2\bigg)^{1/2}\\
&\leq&C_r\bigg(\sum_{n=0}^{\infty}|a_n|^2\bigg)^{1/2},
\end{eqnarray*}
where the last $\leq$ is  due to  Lemma \ref{radiusGaussian}.
So $\big\{\mathcal{R}^{(r)}, \ 0<r<1\big\}$ is a collection of continuous linear operators from $H^2$ to $L^2(\Omega, H^p).$
Write $\tilde{f}_\omega(\theta)$ as the boundary value function of $\sum_{n=0}^{\infty}a_nX_nz^n.$
Combining with the definition of $H^p,$ we have
$$
\bigg(\int_{0}^{2\pi}\bigg|
\sum_{n=0}^{\infty}a_nX_nr^ne^{in\theta}- \tilde{f}_\omega(\theta)\bigg|^p\frac{d\theta}{2\pi}\bigg)^{1/p}\leq2
\bigg(\int_{0}^{2\pi}\big| \tilde{f}_\omega(\theta)\big|^p\frac{d\theta}{2\pi}\bigg)^{1/p}
$$
for all $0<r<1.$
Then, by (\ref{Fernique111}) which implies that the above right-hand side belongs to $L^2(\Omega)$, and  Fatou's lemma,
\begin{eqnarray*}
\limsup_{r\to1}\big|\big|\mathcal{R}^{(r)}f-\mathcal{R}f\big|\big|_{L^2(\Omega, H^p)}^2
\leq\mathbb{E}\bigg(\limsup_{r\to1}\bigg(\int_{0}^{2\pi}\bigg|
\sum_{n=0}^{\infty}a_nX_nr^ne^{in\theta}- \tilde{f}_\omega(\theta)\bigg|^p\frac{d\theta}{2\pi}\bigg)^{2/p}\bigg),
\end{eqnarray*}
which is equal to 0.
 So
$
\mathcal{R}f=\lim_{r\to1}\mathcal{R}^{(r)}f \ \ (\text{in $L^2(\Omega, H^p)$}).
$
Hence, by the uniform boundedness principle, we conclude that $\mathcal{R}$ is bounded.
For  $p>2,$ it is sufficient to show that  $\mathcal{R}: H^2\to L^p(\Omega, H^p)$ is bounded.
Then the proof   is similar to that of $1\le p \le 2$, and is skipped.

\

\noindent
\emph{Proof of (2) of Theorem \ref{mainresult}:} For each polynomial $\sum_{k=0}^{n}a_kz^k,$ if $1\leq p<2,$ then by Lemma \ref{equivalentmoment},

\begin{eqnarray*}
\bigg|\bigg|\sum_{k=0}^{n}a_kX_kz^k\bigg|\bigg|_{L^2(\Omega, H^p)}&\leq&
C_p\bigg|\bigg|\sum_{k=0}^{n}a_kX_kz^k\bigg|\bigg|_{L^p(\Omega, H^p)}\\
&\leq& C_p'\bigg(\int_{0}^{2\pi}\bigg(\mathbb{E}\bigg|\sum_{k=0}^{n}
a_kX_ke^{ik\theta}\bigg|^2\bigg)^{p/2}\frac{d\theta}{2\pi}\bigg)^{1/p}\\
&\leq&C_1\bigg(\sum_{k=0}^{n}|a_k|^2\bigg)^{1/2},
\end{eqnarray*}
where the last ``$\leq$" is due to  Proposition \ref{Bessel}.
Similarly, if $2\leq p<\infty,$ then
\begin{eqnarray*}
\bigg|\bigg|\sum_{k=0}^{n}a_kX_kz^k\bigg|\bigg|_{L^2(\Omega, H^p)}\leq
\bigg|\bigg|\sum_{k=0}^{n}a_kX_kz^k\bigg|\bigg|_{L^p(\Omega, H^p)}
\leq C_2\bigg(\sum_{k=0}^{n}|a_k|^2\bigg)^{1/2}.
\end{eqnarray*}
 Now a density argument enables us to conclude that $\mathcal{R}: H^2\to L^2(\Omega, H^p)$ is bounded.
\end{proof}

\noindent \emph{Remarks.}
\begin{itemize}
  \item For (2), if the Bessel bound of $\{X_n\}_{n=0}^{\infty}$ is $C_{\text{BES}},$ then we have
\[
||\mathcal{R}||\leq\left\{
\begin{array}{ll}
2\sqrt{2C_{\text{BES}}} \ & \ 1\leq p<2;\\
\frac{2\sqrt{C_{\text{BES}}}\big(\Gamma\big(\frac{p+1}{2}\big)\big)^{1/p}}{\pi^{\frac{1}{2p}}} \ & \ 2\leq p<\infty.
\end{array}
\right.
\]

\item

If $p\geq2$ and $\mathcal{R}f\in H^p$ a.s., then there exists a constant $C_1>0$ such that
$$
||\mathcal{R}f||_{L^2(\Omega, H^p)}\geq C_1 \inf_{n\geq0} \big(\mathbb{E}X_n^2\big)^{1/2}||f||_{H^2};
$$
  If $1\leq p\leq2,$ then there exists a constant $C_2$ such that
$$
||\mathcal{R}f||_{L^2(\Omega, H^p)}\leq C_2 \sup_{n\geq0} \big(\mathbb{E}X_n^2\big)^{1/2}||f||_{H^2}.
$$
We can take $C_1=\frac{1}{4}, \ C_2=2\sqrt{2}.$
 By the equivalence of all moments of Gaussian variables, we can replace $L^2(\Omega, H^p)$ by $L^q(\Omega, H^p)$ for any $q\geq1.$
  It is not clear to us whether the above estimates hold for all $p.$ Our restriction on $p$ is due to the use of Jensen's inequality.

\item  For $p=2$,
$\mathcal{R}: H^2\to L^2(\Omega, H^2)$ is bounded if and only if
$\{\mathbb{E}X_n^2\}_{n=0}^\infty \in\ell^\infty.$
Indeed, $$||\mathcal{R}||=\sup_{n\ge 0}\Big(\mathbb{E}X_n^2\Big)^{1/2}.$$

\end{itemize}

\section{Exponential estimates}

\noindent
The following proof of the Paley-Zygmund-type exponential estimates (Proposition \ref{exponentintegral}) is modified from the corresponding arguments for the independent case (\cite{RandomSeries}, pp. 51-52), hence we only outline where the arguments are different.

\noindent
\begin{proof}[Proof of Proposition \ref{exponentintegral}:]
Since $K$ is bounded on $\ell^2,$ we have
$\sup_{n\geq0}\mathbb{E}X_n^2<\infty$ by Theorem \ref{mainresult}. It follows that $\sum_{n=0}^{\infty}a_nX_ne^{in\theta}\in L^2[0,2\pi]$ a.s. So for almost all $\theta,$ $\sum_{n=0}^{\infty}a_nX_ne^{in\theta}$ converges a.s., which is due to Fubini. For any real sequence  $\{\psi_n\}_{n=0}^{\infty}$,
\begin{eqnarray*}
\mathbb{E}\exp\bigg(\la\bigg|\sum_{n=0}^{\infty}a_nX_n\cos(n\theta+\psi_n)\bigg|^2\bigg)
\leq\sum_{n=0}^{\infty}\frac{\la^n\cdot (2n-1)!!\cdot||K||^{2n}}{n!}\bigg(\sum_{k=0}^{\infty}|a_k|^2\bigg)^n<\infty,
\end{eqnarray*}
if $\la<\frac{1}{2||K||^2\sum_{k=0}^{\infty}|a_k|^2}.$ Here we use the Bessel sequence (Proposition \ref{Bessel}).
By Fubini again,
  \begin{equation}\label{12}
  \int_{0}^{2\pi}\exp\bigg(\la\bigg|\sum_{n=0}^{\infty}a_nX_n\cos(n\theta+\psi_n)\bigg|^2\bigg)d\theta<\infty \ \ \text{a.s.}
  \end{equation}
  whenever $\la<\frac{1}{2||K||^2\sum_{k=0}^{\infty}|a_k|^2}.$
For any $\la>0,$  (\ref{12}) still holds by a trick  used in (\cite{RandomSeries}, p. 52).
\end{proof}

\section{Examples}\label{S:E}

\noindent Various examples are collected in this section to illustrate the scope of Theorem \ref{T:main}.

\subsection{Band matrix}

 \noindent The case of  band-limited Gaussian processes admits a neat answer. This clearly generalizes Littlewood's original theorem as well. It is  interesting to observe that only the main diagonal matters.

\begin{proposition}\label{P:band}
Let $\{X_n\}_{n=0}^{\infty}$ be a band-limited, centered Gaussian process. That is, its covariance matrix
      $
       K=\big(\mathbb{E}(X_{i}X_{j})\big)_{i, \ j\geq0}
       $
is a band matrix, with some positive integer $M$ such that
      $
     \mathbb{E}(X_{i}X_{j})=0
     $
     when $
      |i-j|\geq M.$
 Then
    $\{X_n\}_{n=0}^{\infty}$ has the L-property if and only if
      $
  \{\mathbb{E}X_n^2\}_{n=0}^{\infty}\in\ell^\infty.
  $
\end{proposition}

\noindent We shall need the following criteria; its  proof is simply by Schur's test and Theorem \ref{T:main}.

\begin{lemma}\label{Sufficient1}
Let $X_n=\sum_{k=0}^{n}b^n_k\xi_k, \ n\geq0$,  be the canonical representation of a centered Gaussian process, where $\{\xi_k\}_{k\ge 0}$ is a standard Gaussian sequence. If there exists a sequence of real numbers $\eta_n, \ n\geq0$ such that for any $0 \le k \le n$,
$|b_k^n|\leq\eta_{n-k}  \ \ \text{and} \ \ \sum_{n=0}^{\infty}\eta_n<\infty,$
then $\{X_n\}_{n=0}^{\infty}$ has the L-property.
\end{lemma}

\begin{proof}[Proof of Proposition \ref{P:band}]
The necessity  is due to Proposition \ref{necessary2}. For sufficiency, we observe that it is sufficient to consider the case that the covariance matrix is strictly positive definite. Otherwise, let $\{\xi_n\}_{n=0}^{\infty}$ be a standard Gaussian sequence independent with  $\{X_n\}_{n=0}^{\infty}$.
   We   consider
  $
  Y_n=X_n+\epsilon \xi_n
  $ for small $\epsilon>0$, which has a strictly positive definite covariance matrix. A useful fact is that  if $X=\{X_n\}_{n=0}^{\infty}$ and $Y=\{Y_n\}_{n=0}^{\infty}$ are independent centered Gaussian processes, then
$X$ and $Y$ have the $L$-property if and only if $X+Y$ has the $L$-property.

\smallskip

\noindent
Consider a canonical representation
$
X_n=\sum_{k=0}^{n}b_k^n\xi_k, \ n\geq0,
$ where $\xi_n, \ n\geq0$ are independent standard Gaussian variables.
The strictly positive definiteness assumption for the covariance matrix implies that $
b_k^k\neq0, \ k\geq0.$ Some linear algebra shows that $
b_j^k=0, \ 0\leq j\leq k-M$
for each $k\geq M.$
Combining with the assumption, we have
$|b_j^k|\leq\sup_{n\geq0}\big(\mathbb{E}X_n^2\big)^{1/2}, \ j,k\geq0.$
Then, by Lemma \ref{Sufficient1} (here, take $\eta_k=\sup_{n\geq0}\big(\mathbb{E}X_n^2\big)^{1/2}$ if $0\leq k\leq M-1;$ $\eta_k=0$ if $k\geq M$),
   $\{X_n\}_{n=0}^{\infty}$ has the $L$-property.
\end{proof}

\subsection{Stationary process}

\noindent
A (centered) Gaussian process $\{X_n\}_{n\in\mathbb{Z}}$ is called \textbf{stationary}, if for any  $n, h$ and $k,$ two random vectors
$\big(X_n, X_{n+1}, \ldots, X_{n+k-1}\big) $ and $ \big(X_{n+h}, X_{n+h+1}, \ldots, X_{n+h+k-1}\big)$
have the same distribution. Now let $X=\{X_n\}_{n\in\mathbb{Z}}$ be a stationary Gaussian process. Denote by
$
H_n(X)
$
the closed subspace of $L^2(\Omega)$ spanned by $\{X_k; \ k\leq n\}.$ If
$
\cap_{n\in\mathbb{Z}}H_n(X)\subset\mathbb{R},
$ then $X$ is said to be \textbf{purely nondeterministic}.

\begin{proposition}
If a centered Gaussian process $X=\{X_n\}_{n\in\mathbb{Z}}$ is stationary, purely nondeterministic and Markov, then $\{X_n\}_{n=0}^{\infty}$ has the L-property.
\end{proposition}

\begin{proof}
By our assumption on $X$ and by the corollary in   (\cite{Hida}, p. 42),
$X$ has the form
$
X_n=\sum_{j=-\infty}^{n}ac^{n-j}\xi_j, \ \ n\in\mathbb{Z}, \ 0<|c|<1.
$ The system
 $\big\{ac^{n-j}, \ j\leq n; \ \xi_j, \ j\in\mathbb{Z}\big\}$
is the canonical representation of $X.$ Then one can apply Schur's test or  Proposition \ref{Sufficient1} to conclude that the covariance matrix of $\{X_n\}_{n=0}^{\infty}$ is bounded on $\ell^2$, hence the $L$-property by Theorem \ref{T:main}.
\end{proof}

\noindent \emph{Remark.} By the Toeplitz Theorem \cite{Toeplitz11} (see  \cite{Grudsky} also),   if $\{X_n\}_{n=0}^{\infty}$ is a centered stationary Gaussian process, then its covariance matrix $K$ is bounded on $\ell^2$ if and only if there exists some function $h\in L^\infty(\mathbb{T})$ such that the Fourier series of $h$ is $\sum_{n=-\infty}^{\infty}\mathbb{E}(X_0X_{|n|})e^{in\theta}.$

\subsection{Lacunary}

\noindent
A sequence $\{n_k\}_{k\ge 1}$ of positive integers is called \textbf{lacunary} if
$
\inf_{k\geq1}\frac{n_{k+1}}{n_k}>1.
$
The corresponding series $\sum_{k=1}^{\infty}a_{n_k}z^{n_k}$ is called a lacunary series.

\begin{proposition}
Let $\{X_n\}_{n=0}^{\infty}$ be a centered Gaussian process with $\sup_{n\geq0}\mathbb{E}X_n^2<\infty.$ Then for any $\{a_n\}_{n=0}^{\infty}\in\ell^2$ which is  lacunary for $n$, we have
$
\sum_{n=0}^{\infty}a_nX_nz^n\in  H^p \ \ \text{a.s.}$
for any $p \ge 1$.
\end{proposition}
\noindent This has the following two consequences:
\begin{itemize}
  \item If a centered Gaussian process $X=\{X_n\}_{n=0}^{\infty}$ is a lacunary sequence, then
      $X$ has the $L$-property if and only if $\sup_{n\geq0}\mathbb{E}X_n^2<\infty.$
  \item Let $\{X_n\}_{n=0}^{\infty}$ be independent centered Gaussian variables. Let
\begin{equation}
E=\Big\{\{a_n\}_{n=0}^{\infty}\in\ell^2: \ \{a_n\}_{n=0}^{\infty} \ \text{is a lacunary sequence}\Big\}.
\end{equation}
Then for any $\{a_n\}_{n=0}^{\infty}\in E,$ $\sum_{n=0}^{\infty}a_nX_nz^n\in H^p$ a.s. for $p\ge 1$ if and only if $\{ \mathbb{E}X_n^2\} \in l^\infty.$
\end{itemize}

\smallskip

\noindent For proof, observe that if $\{a_n\}_{n=0}^{\infty}\in\ell^2$ is  lacunary, then by Theorem 6.2.2 in \cite{2016Taylor},
$
\sum_{n=0}^{\infty}a_nX_nz^n\in H^p \ \ \text{a.s.}
$ if and only if
$ \sum_{n=0}^{\infty}|a_nX_n|^2<\infty \ \ \text{a.s.}
$
This is due to
\begin{eqnarray*}
\mathbb{E}\Big(\sum_{n=0}^{\infty}|a_nX_n|^2\Big)=\sum_{n=0}^{\infty}|a_n|^2\mathbb{E}X_n^2
\leq\sup_{n\geq0}\mathbb{E}X_n^2\sum_{n=0}^{\infty}|a_n|^2<\infty.
\end{eqnarray*}

\subsection{The equivalence of Gaussian processes}

\noindent
Let $\mu_i$  be the induced Gaussian measures on $\big(\mathbb{R}^{\mathbb{N}\cup \{0\}}, \mathcal{E}^{\mathbb{N}\cup \{0\}}\big)$ via two Gaussian processes $X^{(i)}=\{X^{(i)}_n\}_{n=0}^{\infty}$, $i=1, 2$.

%

\begin{proposition}
There exist two centered Gaussian processes $X^{(1)}, X^{(2)}$ such that both of them have the L-property but their corresponding  Gaussian measures $\mu_1, \mu_2$ are mutually singular.
\end{proposition}

\noindent Indeed, let
$A=\big(a_{i,j}\big)_{i,j\geq0},$ where $
a_{i,j}=\left\{
\begin{array}{ll}
2^{j-i} \ & \ i\geq j\\
0 & \ i<j
\end{array}
\right.
$
and let 
$
X_n=\sum_{k=0}^{n}a_{n,k}\xi_k, \ n\geq0,
$
where $\{\xi_n\}_{n=0}^\infty$ is a standard Gaussian sequence.
Then, $\{X_n\}_{n=0}^{\infty}$ is a centered Gaussian process with the $L$-property. This  is due to the Schur test, the fact that $
K=\big(\mathbb{E}(X_mX_n)\big)_{m,n\geq0}= AA^T$, and Theorem \ref{T:main}. The induced Gaussian measure is singular to that of a standard Gaussian sequence by an application of Theorem 6.2 in \cite{Hida}.

\subsection{Necessary conditions}

\noindent
In this subsection we present two necessary conditions for a centered Gaussian process to have the $L$-property. 

\begin{proposition}\label{necessary2}
Let $\{X_n\}_{n=0}^{\infty}$ be a centered Gaussian process with the L-property.
  Then for any $m\geq0,$
 $
  \big\{\mathbb{E}(X_nX_m)\big\}_{n=0}^{\infty}\in\bigcap_{1\leq p<\infty}(H^2, H^p).
  $
Furthermore, if $\Big\{\big|\mathbb{E}(X_nX_m)\big|\Big\}_{n=0}^{\infty}$ is  decreasing  for some $m,$ then
 for every $\epsilon>0,$
$$\mathbb{E}(X_nX_m)=o\big(n^{-\frac{1}{2}+\epsilon}\big).$$
\end{proposition}

\begin{proof}
 Let $p\in[1, \infty).$ For  $f(z)\in H^2,$
 by Fernique's theroem,
$
\mathbb{E}\exp\big(\alpha||\mathcal{R}f||_{H^p}^2\big)<\infty
$ for some
$\alpha>0.$
Write
$
X=\sum_{n=0}^{\infty}a_nX_nz^n
$
and let $\tilde{f}_\omega(\theta)$ denote the boundary value function of $\sum_{n=0}^{\infty}a_nX_nz^n.$
Then
$
\mathbb{E}\big(||X_m\cdot X||_{H^p}\big) \leq\big(\mathbb{E}|X_m|^2\big)^{\frac{1}{2}}\big(\mathbb{E}||X||_{H^p}^2\big)^{\frac{1}{2}},
$ which is finite by Fernique's theorem.
Let $0<r<1$ and let $
g_r=\sum_{n=0}^{\infty}a_nX_mX_nr^nz^n$.
By Fatou's lemma,
\begin{eqnarray*}
\limsup_{r\to1}\big|\big|g_r-X_m\cdot X\big|\big|_{L^1(\Omega, H^p)}
&=&\limsup_{r\to1}\mathbb{E}\Big(|X_m|\Big(\int_{0}^{2\pi}\Big|
\sum_{n=0}^{\infty}a_nX_nr^ne^{in\theta}- \tilde{f}_\omega(\theta)\Big|^p\frac{d\theta}{2\pi}\Big)^{1/p}\Big)\\
&\leq&\mathbb{E}\Big(\limsup_{r\to1}|X_m|\Big(\int_{0}^{2\pi}\Big|
\sum_{n=0}^{\infty}a_nX_nr^ne^{in\theta}- \tilde{f}_\omega(\theta)\Big|^p\frac{d\theta}{2\pi}\Big)^{1/p}\Big),
\end{eqnarray*}
which is zero. So
$
\mathbb{E}(X_m\cdot X)=\sum_{n=0}^{\infty}\mathbb{E}(X_mX_n)a_nz^n\in H^p.$
Hence, $
  \big\{\mathbb{E}(X_nX_m)\big\}_{n=0}^{\infty}\in(H^2, H^p).$

\noindent
Next we assume that $\big\{|\mathbb{E}(X_nX_m)|\big\}_{n=0}^{\infty}$ is a decreasing sequence for some $m.$ For convenience, we write $
    \la_n\triangleq\mathbb{E}(X_nX_m), \ n\geq0.$
Fix $\epsilon>0$. Note that $
    f(z)=\sum_{n=0}^{\infty}\frac{\text{sgn}(\la_n)}{n^{\frac{1}{2}+\epsilon}}z^n \in H^2$ and
$
 \sum_{n=0}^{\infty}\frac{\la_n\cdot\text{sgn}(\la_n)}{n^{\frac{1}{2}+\epsilon}}z^n
=\sum_{n=0}^{\infty}\frac{|\la_n|}{n^{\frac{1}{2}+\epsilon}}z^n\in H^p,
$
because of $\{\la_n\}_{n=0}^{\infty}\in(H^2, H^p).$
Since $\Big\{\frac{|\la_n|}{n^{\frac{1}{2}+\epsilon}}\Big\}_{n=1}^{\infty}$ is also  decreasing,
 by  Theorem 6.2.14 (\cite{2016Taylor},  p. 121),
we have
$
\sum_{n=1}^{\infty}(n+1)^{p-2}\Big(\frac{|\la_n|}{n^{\frac{1}{2}+\epsilon}}\Big)^p<\infty.
$
\end{proof}

\noindent
The following necessary condition is curious to us. Although it seems strong, we do not know how to make further use of it.
\begin{proposition}
Let $\{X_n\}_{n=0}^{\infty}$ be a centered Gaussian process defined on a probability space $(\Omega, \mathcal{F}, P)$. Assume that its covariance matrix $K$ is bounded on $\ell^2.$ Then, for each $X\in L^2(\Omega)$  and $\{a_n\}_{n=0}^{\infty}\in\ell^2,$ we have
$
\mathbb{E}\big(X\cdot\sum_{n=0}^{\infty}a_nX_nz^n\big)\in\mathcal{W},
$
where $\mathcal{W}$ is the Wiener algebra.
\end{proposition}

\begin{proof}
Write
$
Y=\sum_{n=0}^{\infty}a_nX_nz^n.
$
Since $K$ is bounded on $\ell^2$, we have
$\sup_{n\geq 0}\mathbb{E}X_n^2<\infty.$
 By direct calculation,
$
\mathbb{E}\big(||X\cdot Y||_{H^2}\big)\leq\big(\mathbb{E}|X|^2\big)^{1/2}\Big(\mathbb{E}\big(\sum_{n=0}^{\infty}|a_nX_n|^2\big)\Big)^{1/2}<\infty.
$
By the vector-valued dominated convergence theorem,
$
\mathbb{E}(X\cdot Y)
=\sum_{n=0}^{\infty}\mathbb{E}(XX_n)a_nz^n.
$
Since the covariance matrix $K$ is bounded on $\ell^2,$ by Proposition \ref{Bessel},   $\{X_n\}_{n=0}^{\infty}$ is a Bessel sequence. Consequently,  $\big\{E(XX_n)\big\}_{n=0}^{\infty}\in\ell^2.$ Hence,
$
\sum_{n=0}^{\infty}\big|\mathbb{E}(XX_n)a_n\big|<\infty
$
and
$
\mathbb{E}\big(X\cdot\sum_{n=0}^{\infty}a_nX_nz^n\big)\in\mathcal{W}.
$
\end{proof}

\subsection{Sidon set}

\noindent
The following is an example such that
$\sum_{n=0}^{\infty}a_nX_nz^n\in H^p$ a.s. for any $p>1$, but $\sum_{n=0}^{\infty}a_nX_nz^n\notin H^\infty$ a.s., based on Exercise 4 in (\cite{RandomSeries}, p. 232). Let $\{\xi_n\}_{n=0}^{\infty}$ be a standard Gaussian sequence. Let $\{a_n\}_{n=0}^{\infty}\in\ell^2$ be such that $a_n=0$ except when $n$ belongs to {a Sidon set}. Here, a set of integers $\Lambda$ is called a \textbf{Sidon set} if $\Lambda$ satisfies the following property: if $g\in L^\infty[0,2\pi]$ and the Fourier series of $g$ is $\sum_{\lambda\in\Lambda}d(\lambda)e^{i\lambda\theta},$ then $\sum_{\lambda\in\Lambda}\big|d(\lambda)\big|<\infty$ (\cite{RandomSeries}, p. 71). Under the above conditions, we have
$
 \sum_{n=0}^{\infty}a_n\xi_nz^n\in H^p \ \ \text{a.s.} $ for any $p >1$.
Furthermore,
 if $\sum_{n=0}^{\infty}|a_n|=\infty,$ then
 $
 \sum_{n=0}^{\infty}a_n\xi_nz^n\notin H^\infty\ \ \text{a.s.}
 $
 This follows from a celebrated condition of Marcus-Pisier
(\cite{RandomSeries}, pp. 231-232). (For details, see Exercise 4 in (\cite{RandomSeries}, p. 232)). Another example with the above property has been introduced by Paley and Zygmund (\cite{Paley1}, p. 350).

\subsection{The sharpness of Theorem \ref{T:main}}\label{SS:sharp}
The purpose of this subsection is to discuss the sharpness of  Theorem  \ref{T:main} by a sequence of $\{X_n\}_{n\ge 0}$ with maximal dependence. In particular, we seek to  answer the following three questions:
\begin{itemize}
\item[(1)] Is $\sup_{n\geq0}\mathbb{\mathbb{E}}X_n^2 <\infty$ sufficient for the $L$-property?
\item[(2)] Is the boundedness of covariance matrix $K$ necessary for the $L$-property?
\item[(3)] Is the covariance structure of $K$ corresponding to the $L_p$-property different as $p$ varies?
\end{itemize}

\noindent
Let $\xi$ be a standard Gaussian variable and let $\{c_n\}_{n=0}^{\infty}$ be a sequence of complex scalars. In this subsection we consider the case $X_n=c_n \xi$.  

\bigskip

\noindent 
The answer to (1) is no by taking $c_n=1$. That is, $E(X_n X_m)=1$ for all $n, m \geq0.$ In this case we choose  $f(z)=\sum_{n=0}^{\infty}a_nz^n\in H^2\setminus\big(\cup_{p>2}H^p\big).$ Then
$(\mathcal{R}f)(z)=\xi\cdot\sum_{n=0}^{\infty}a_nz^n\notin\cup_{p>2}H^p$ a.s.

\bigskip

\noindent
The answer to (2) is no by taking $c_n=\frac{1}{\sqrt{n+1}}$.
The covariance matrix $$K=\Big(\mathbb{E}\big(X_nX_m\big)\Big)_{n, m\geq0}=\bigg(\frac{1}{\sqrt{(n+1)(m+1)}}\bigg)_{n, m\geq0}$$ is clearly unbounded.  We claim that it has the $L_p$-property for each $p\geq2.$ In fact, $\{c_n\}_{n=0}^{\infty}$ is a coefficient multiplier from $H^2$ to $H^p$ for any $p \in (2, \infty)$ by a result of Duren (\cite{Duren69}, Theorem 1).  

\smallskip

\noindent
The answer to (3) is yes. If $2\leq p_1<p_2,$ then by \cite{Duren69} again, there exists some  $c_n=O\big( n^{-(\frac{1}{2}-\frac{1}{p_1})}\big) $ such that    $\{c_n\}_{n=0}^{\infty}$  is a coefficient multiplier from $H^2$ to $H^{p_1},$ but not to $H^{p_2}$. This implies that  there exists  $g(z)=\sum_{n=0}^{\infty}b_nz^n\in H^2$ such that $\sum_{n=0}^{\infty}b_nc_nz^n\notin H^{p_2}.$ So $\{X_n\}_{n=0}^{\infty}$ has the $L_{p_1}$-property but no  $L_{p_2}$-property.

\

\noindent
Now it comes  a natural question: what might be the ``correct" characterization of the $L$-property in terms of a Gaussian process?
In principle everything about a Gaussian process is encoded in its mean and covariance functions, but this can be hard to carry out in certain situations. It is not clear to us whether a clean characterization of the $L$-property in terms of the covariance exists at all, although we are somehow pessimistic.
In Subsection 5.4 we shall see that bandlimitedness of the covariance operator is sufficient for the $L$-property.
Together with the above $c_n=1$ example, it suggests that a precise characterization in terms of $K$, if ever to be found,
should be some type of decay condition for entries far off the diagonal.

\

\noindent
On the other hand, the above example of $c_n=\frac{1}{\sqrt{n+1}}$ shows that the boundedness of $K$ as a sufficient
condition is borderline optimal at least in certain sense. Indeed, if $c_n=\frac{1}{(n+1)^{1/2+\epsilon}}$ for  $\epsilon >0$, then $K$ is bounded by Bessel sequences. On the other hand, for any $\epsilon >0$, there exists a sign sequence $\ep_n \in \{1, -1\}$ such that  $\{c_n=\epsilon_n \frac{1}{(n+1)^{1/2-\epsilon}}\}_{n=0}^\infty$  is not  a coefficient multiplier from $H^2$ to $H^q$ for $q$ large enough; see the arguments in the proof of Theorem 3 in \cite{Duren69}. In particular, $\{X_n\}_{n=0}^\infty$ has no $L$-property.

\

\noindent
In summary,  the full characterization of the $L$-property in terms of $K$, if ever to be found,
should be a class of $K$ only ``slightly larger" than the class of bounded operators. We have no clue  how to guess such a class in a precise way, although we have the vague idea that this might be comparable to the fact that $\cap_{p\ge 1} L^p(\mathbb{T})$ is slightly larger than $C(\mathbb{T})$.

\bignobf{Acknowledgement}

\noindent
Cheng is supported by NSFC (11871482). Fang is supported by MOST of Taiwan (106-2115-M-008-001-MY2) and NSFC (11571248) during his visits to   Soochow University in China. Guo and Liu are  supported by NSFC (11871157). Part of this work was done during Liu's visit to NCU in Taiwan in the Fall semester of 2018.


\bigskip

\noindent{\small SCHOOL OF MATHEMATICS SCIENCE, DALIAN UNIVERSITY OF TECHNOLOGY, DALIAN 116024, P. R. CHINA
\noindent Email address: gzhcheng@dlut.edu.cn}

\bigno

\noindent{\small DEPARTMENT OF MATHEMATICS, NATIONAL CENTRAL UNIVERSITY, CHUNGLI, TAIWAN (R.O.C)
\noindent Email address: xfang@math.ncu.edu.tw}

\bigno

\noindent{\small SCHOOL OF MATHEMATICAL SCIENCES, FUDAN UNIVERSITY, SHANGHAI, 200433, P. R. CHINA
\noindent Email address: kyguo@fudan.edu.cn}

\bigno

\noindent{\small SCHOOL OF MATHEMATICS SCIENCE, DALIAN UNIVERSITY OF TECHNOLOGY, DALIAN 116024, P. R. CHINA
\noindent Email address: 2020024050@dlut.edu.cn}

\end{document}